\documentclass[11pt]{article}
\usepackage{amssymb,amsmath,amsfonts}
\usepackage{stmaryrd}
\usepackage{color}
\usepackage[latin1]{inputenc}
\usepackage{multirow}
\usepackage{subfigure}
\usepackage{epsfig}
\usepackage{epic}
\usepackage{pstricks}
\usepackage{pst-node}
\usepackage{pstricks-add}

\def\BBox{\kern  -0.2cm\hbox{\vrule width 0.2cm height 0.2cm}}

\newtheorem{theorem}{Theorem}[section]
\newtheorem{definition}{Definition}[section]

\newtheorem{proposition}{Proposition}[section]

\title{A construction of small $(q-1)$-regular graphs of girth 8}

\author{ M. Abreu$^{1}$, G. Araujo-Pardo$^{2}$, C. Balbuena$^{3}$,
D. Labbate$^{1}$
\thanks{ Research   supported by the Ministerio de Educación y Ciencia,
Spain, the European Regional Development Fund (ERDF) under project
MTM2011-28800-C02-02; by the Catalonian Government
under project 1298 SGR2009; by CONACyT-México under project 57371;
by PAPIIT-México under project 104609-3; 
by the Italian Ministry MIUR 
and carried out within the activity of INdAM-GNSAGA. \newline \footnotesize{\em
Email addresses:} marien.abreu@unibas.it (M. Abreu),~
garaujo@matem.unam.mx (G. Araujo), ~ m.camino.balbuena@upc.edu (C.
Balbuena), \, \, ~domenico.labbate@unibas.it (D. Labbate)}
 \\[2ex]
$^1${\footnotesize Dipartimento di Matematica, Informatica ed Economia,  Università degli Studi della
Basilicata,}\\
{\footnotesize Viale dell'Ateneo Lucano, I-85100 Potenza, Italy.} \\$^2$
{\footnotesize Instituto de Matem\'{a}ticas, Universidad Nacional Autónoma de México,} \\
{\footnotesize México D. F., México }\\
$^3${\footnotesize Departament de Matemática Aplicada III, Universitat
Politècnica de Catalunya, }\\
{\footnotesize Campus Nord, Edifici C2, C/ Jordi Girona 1 i 3 E-08034 Barcelona,
Spain.}
}

\date{}

\begin{document}
\maketitle
\begin{abstract}
In this note we  construct a new infinite family of $(q-1)$-regular graphs of girth $8$ and order $2q(q-1)^2$
for all prime powers $q\ge 16$, which are the smallest known so far whenever $q-1$ is not a prime power or a prime power plus one itself.
\end{abstract}




\section{Introduction}

Throughout this note, only undirected simple graphs without loops
or multiple edges are considered. Unless otherwise stated, we
follow  the book by Bondy and Murty \cite{BM} for terminology and notation.

 Let $G$ be a graph with vertex set
$V=V(G)$ and edge set
  $E=E(G)$.  The \emph{girth}  of a graph $G$ is the number $g=g(G)$ of edges in a
smallest cycle. For every $v\in V$, $N_G(v)$ denotes the \emph{neighbourhood} of $v$,
that is, the set of all vertices adjacent to $v$. The \emph{degree} of a vertex $v\in V$ is the
cardinality of    $N_G(v)$.   A graph is called
\emph{regular} if all the vertices have the same degree. A  \emph{$(k,g)$-graph} is a  $k$-regular graph with girth $g$.  Erd\H os and Sachs
 \cite{ES63}     proved the existence of  $(k,g)$-graphs
 for all values of $k$ and $g$ provided that $k \ge 2$. Since then most work carried
 out has  focused on constructing a smallest one (cf. e.g.
 \cite{AFLN06,AFLN08,AABL12,ABH10,BI73,B08,B09,B66,E96,GH08,LUW97,M99,OW81}).
A  \emph{$(k,g)$-cage} is a  $k$-regular graph with girth $g$ having the smallest possible
number of vertices. Cages have been intensely studied
since they were introduced by
Tutte \cite{T47} in 1947. More details about constructions of cages can be found in the
recent survey by Exoo and Jajcay \cite{EJ08}.

In this note we are interested in $(k,8)$-cages.
Counting the number of vertices in the
distance partition with respect to an edge yields the following lower bound
on the order of a $(k,8)$-cage:

  \begin{equation}\label{lower} n_0(k,8) = 2(1+(k-1)+(k-1)^2+(k-1)^3).\end{equation}

A  $(k,8)$-cage with
 $n_0(k,8)$ vertices is called a Moore \emph{$(k,8)$-graph} (cf. \cite{BM}). These graphs have been constructed as  the incidence graphs of generalized quadrangles of order $k-1$ (cf. \cite{B66}).          All these objects   are   known
to exist for all prime power values of $k-1$ (cf. e.g. \cite{ B97, GR00}), and no example is known when $k-1$ is not a prime power.
Since they are incidence graphs, these cages are bipartite and have diameter $4$.

A subset $U\subset V(G)$ is said to be   \emph{a perfect  dominating set of $G$} if for each vertex $x\in V(G)\setminus U$,  $|N_G(x)\cap U|=1$ (cf. \cite{HHS98}). Note that if $G$ is a $(k,8)$-graph and $U$ is a  perfect dominating set of $G$, then $G-U$ is clearly a $(k-1,8)$-graph.
Using classical generalized quadrangles, Beukemann and   Metsch  \cite{BM10} proved that the
cardinality of a perfect dominating set $ B $ of a Moore $(q+1,8)$-graph, $q$ a prime power, is at most $|B|\le 2(2q^2+2q)$ and if $q$ is even
$|B|\le 2(2q^2+q+1)$.

For $k=q+1$ where $q\ge 2$ is a prime power, we find a   perfect dominating set of cardinality $2(q^2+3q+1)$ for all $q$  (cf. Proposition \ref{perfect1}). This result allows us to explicitly obtain $q$-regular graphs of girth $8$ and
order $2q(q^2-2)$ for any prime power $q$ (cf. Definition \ref{GrGq}).
Finally, we prove the existence of a perfect dominating set of these $q$-regular graphs  which allow us to construct a new infinite family of $(q-1)$-regular graphs of girth $8$ and order $2q(q-1)^2$ for all prime powers $q$ (cf. Theorem \ref{q-1}), which are the smallest known so far for $q\ge 16$ whenever $q-1$ is not a prime power or a prime power plus one itself. Previously, the  smallest known $(q-1,8)$-graphs, for $q$ a prime power, were those of order $2q (q^2-q -1)$  which appeared in  \cite{B09}. The first ten improved values appear in the following table in which $k=q-1$ is the regularity of a $(k,8)$--graph, and the other columns contain the old and the new upper bound on its order.

\begin{center}
\begin{tabular}{|c|c|c|c|c|c|c|}
  \hline
  k & Bound in \cite{B09} & New bound  & & k & Bound in \cite{B09} & New bound\\
  \hline
  15 &   7648 &   7200 & & 52 & 292030 & 286624 \\
  22 &  23230 &  22264 & & 58 & 403678 & 396952 \\
  36 &  98494 &  95904 & & 63 & 515968 & 508032 \\
  40 & 134398 & 131200 & & 66 & 592414 & 583704 \\
  46 & 203134 & 198904 & & 70 & 705598 & 695800 \\
  \hline
\end{tabular}
\end{center}



\section{Construction of small $(q-1)$-regular graphs of girth 8}\label{SecCons}

 In this section we construct $(q-1)$-regular graphs of girth 8 with $2q(q-1)^2$ vertices, for every prime power $q \geq 4$.
To this purpose we need the following coordinatization of a Moore $(q+1, 8)$-cage $\Gamma_q$.

\begin{definition} \cite{M98,P70} \label{cage}
Let $\mathbb{F}_q$ be a finite  field with $q\ge 2$ a prime power and $\varrho$  a symbol not belonging to $\mathbb{F}_q$. Let $\Gamma_q= \Gamma_q[V_0,V_{1}]$ be a bipartite graph with vertex sets
$V_i=\mathbb{F}_q^3\cup\{(\varrho,b,c)_i, (\varrho,\varrho,c)_i: b, c \in \mathbb{F}_q \}\cup \{(\varrho, \varrho, \varrho)_i\}$, $i=0,1$, and
edge set defined as follows:
$$ \begin{array}{l}\mbox{For all } a\in \mathbb{F}_q\cup \{\varrho\}  \mbox{ and for all } b,c\in \mathbb{F}_q:\\[2ex]
N_{\Gamma_q}((a,b,c)_{1} )= \left\{\begin{array}{ll}
    \{(w,~aw+b,~a^2w+2ab+c)_{0}: w \in \mathbb{F}_q \}\cup \{(\varrho,a,c)_{0} \}  &\mbox{ if }   a\in \mathbb{F}_q;
\\[2ex]
\{( c ,b,w )_{0}: w \in \mathbb{F}_q \}\cup \{(\varrho,\varrho,c)_{0} \}  &\mbox{
if }  a= \varrho.
\end{array}\right.
\\
\mbox{}\\
N_{\Gamma_q}((\varrho,\varrho,c)_{1})= \{(\varrho,c,w)_{0}: w \in \mathbb{F}_q \}\cup
\{(\varrho,\varrho,\varrho)_{0} \}\\[1ex]
N_{\Gamma_q}((\varrho,\varrho,\varrho)_{1})= \{(\varrho,\varrho,w)_{0}: w \in \mathbb{F}_q\} \cup
\{(\varrho,\varrho,\varrho)_{0} \}.
\end{array}
$$
Or equivalently
$$ \begin{array}{l}\mbox{For all } i\in \mathbb{F}_q\cup \{\varrho\}  \mbox{ and for all }  j,k\in \mathbb{F}_q:\\[2ex]
N_{\Gamma_q}((i,j,k)_{0} )= \left\{\begin{array}{ll}
    \{(w,~j-wi,~w^2i-2wj+k)_{1}: w \in \mathbb{F}_q \}\cup \{(\varrho,j,i)_1 \}  &\mbox{ if }   i\in \mathbb{F}_q;
\\[2ex]
\{( j ,w,k )_{1}: w \in \mathbb{F}_q \}\cup \{(\varrho,\varrho,j)_{1} \}  &\mbox{ if }  i= \varrho.
\end{array}\right.
\\
\mbox{}\\
N_{\Gamma_q}((\varrho,\varrho,k)_{0})= \{(\varrho,w,k)_{1}: w \in \mathbb{F}_q \}\cup
\{(\varrho,\varrho,\varrho)_{1} \}; \\[1ex]
N_{\Gamma_q}((\varrho,\varrho,\varrho)_{0})= \{(\varrho,\varrho,w)_{1}: w \in \mathbb{F}_q\} \cup
\{(\varrho,\varrho,\varrho)_{1} \}.
\end{array}
$$
\end{definition}

\noindent Note that $\varrho$ is just a symbol not belonging to $\mathbb{F}_q$ and no arithmetical operation  will be performed  with it.  Figure \ref{spanning} shows a spanning tree of $\Gamma_q$ with the vertices  labelled according to Definition \ref{cage}.
\begin{figure}
  \centering
  \includegraphics[width=16cm]{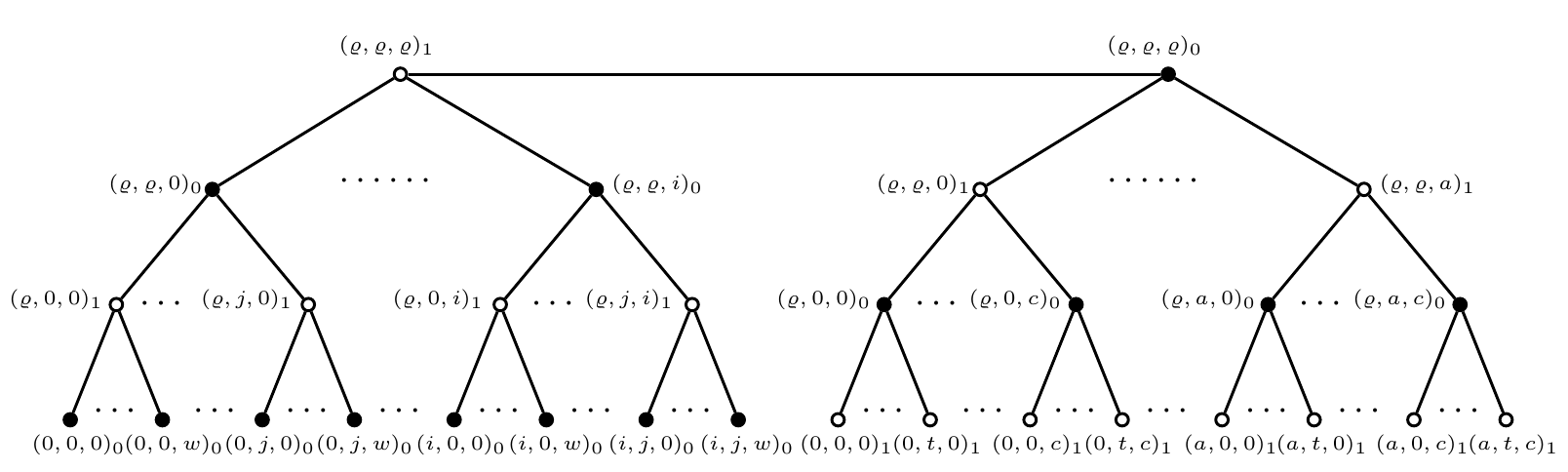}\\
  \caption{Spanning tree of $\Gamma_q$.}\label{spanning}
\end{figure}


 \begin{proposition}\label{perfect1} Let $q\ge 2$ be a  prime power  and let  $\Gamma_q=\Gamma_q[V_{0},V_{1}]$ be the Moore
$(q+1, 8)$-graph with the coordinatization in Definition \ref{cage}.
  Let $A=\{(\varrho,0,c)_{1}:c\in \mathbb{F}_q\}\cup \{(\varrho,\varrho,0)_{1}\}$ and let $x\in   \mathbb{F}_q\setminus \{0\}$.
  Then the set $$N_{\Gamma_q}[A]\cup \left(\bigcap_{a\in A}N_{\Gamma_q}^2(a)\right)\cup N_{\Gamma_q}^2[(\varrho,\varrho,x)_{1}]  $$ is a perfect dominating set of  $\Gamma_q$ of cardinality $2(q^2+3q+1)$.

  \end{proposition}


\begin{proof}
 From Definition \ref{cage}, it follows that
  $A=\{(\varrho,0,c)_1: c\in \mathbb{F}_q\}\cup \{(\varrho,\varrho,0)_1\}$ has cardinality $q+1$ and its elements are mutually at distance four. Then $|N_{\Gamma_q}[A]|=(q+1)^2+q+1$.
By Definition \ref{cage},   $N_{\Gamma_q}((\varrho,0,c)_{1})=\{(c,0,w)_{0}:w\in \mathbb{F}_q\}\cup \{(\varrho,\varrho,c)_{0}\}$; and
$N_{\Gamma_q}((\varrho,\varrho,0)_{1})=\{(\varrho,0,w)_{0}:w\in \mathbb{F}_q\}\cup \{(\varrho,\varrho,\varrho)_{0}\}$. Then $(\varrho,\varrho,\varrho)_1\in N_{\Gamma_q}^2((\varrho,0,c)_{1}))\cap  N_{\Gamma_q}^2((\varrho,\varrho,0)_{1}))$ for all $c\in \mathbb{F}_q$. Moreover,  $N_{\Gamma_q}((c,0,w)_{0})=\{(a,-ac,a^2c+w)_1:a\in \mathbb{F}_q\}\cup \{(\varrho,0,c)_{1}\}$. Thus, for all $c_1,c_2,w_1,w_2\in \mathbb{F}_q$, $c_1\ne c_2$, we have $(a,-c_1a ,a^2c_1+w_1)_1=(a,-c_2a ,a^2c_2+w_2)_1$ if and only if $a=0$ and $w_1=w_2$. Let $I_A=\bigcap_{a\in A}N_{\Gamma_q}^2(a)$. We conclude that $ I_A=\{(\varrho,\varrho,\varrho)_1\}\cup \{(0,0,w)_{1}: w\in \mathbb{F}_q\}$  which implies that $|N_{\Gamma_q}[A]|+|I_A|=(q+1)^2+2(q+1)$.

  Since $N_{\Gamma_q}^2[(\varrho,\varrho,x)_{1}]=\bigcup_{j\in\mathbb{F}_q }
N_{\Gamma_q}[(\varrho,x,j)_{0}]  \cup N_{\Gamma_q}[(\varrho,\varrho,\varrho)_{0}]$ we obtain
   that
$(N_{\Gamma_q}[A]\cup I_A)\cap N_{\Gamma_q}^2[(\varrho,\varrho,x)_{1}]=\{(\varrho,\varrho,\varrho)_{0}, (\varrho,\varrho,0)_{1}, (\varrho,\varrho,\varrho)_{1}\}$.
Let  $D=N_{\Gamma_q}[A]\cup I_A\cup N_{\Gamma_q}^2[(\varrho,\varrho,x)_{1}]$, then
 $$\begin{array}{ll} |D| &=|N_{\Gamma_q}[A]|+|I_A|+|N_{\Gamma_q}^2[(\varrho,\varrho,x)_{1}]|-3\\
 &=(q+1)^2+2(q+1)+  1+(q+1)+q(q+1)-3\\ &=2q^2+6q+2.\end{array}$$
 Let us prove that $D$ is a perfect dominating set of $\Gamma_q$.

Let $H$ denote the subgraph of $\Gamma_q$ induced by $D$. Note that for $t,c\in \mathbb{F}_q$, the vertices $(x,t,c)_1\in N_{\Gamma_q}^2((\varrho,\varrho,x)_{1})$ have degree 2 in $H$ because they are adjacent to the vertex $(\varrho, x, t)_0\in N_{\Gamma_q}(\varrho, \varrho,x)_1$ and also to the vertex  $(-x^{-1}t,0, xt+z)_0\in  N_{\Gamma_q}(A)$. This   implies that the vertices $(i,0,j)_0\in N_{\Gamma_q}(A)$, $i,j\in \mathbb{F}_q$,   have degree 3 in $H$ and, also that the diameter of $H$ is 5. Moreover, for $k\in \mathbb{F}_q$, the vertices $(\varrho, \varrho, k)_0, (\varrho, 0, k)_0  \in D$ have degree 2 in $H$ and the vertices $(\varrho, \varrho, j)_1\in D$, $j\in \mathbb{F}_q\setminus \{0,x\}$ have degree 1 in $H$. All other vertices in $D$ have degree $q+1$ in $H$.

Since the diameter of $H$ is $5$ and the girth is $8$,  $|N_{\Gamma_q}(v)\cap D|\le 1$ for all $v\in V(\Gamma_q)\setminus D$, and also  for all distinct $d,d'\in D$ we have $(N_{\Gamma_q}(d)\cap N_{\Gamma_q}(d'))\cap (V(\Gamma_q)\setminus D)=\emptyset$. Then,
$|N_{\Gamma_q}(D)\cap (V(\Gamma_q)\setminus D)|= q^2 (q-2)+2q(q-1)+(q-2)q+q^2(q-1)=2q^3-4q=|V(\Gamma_q)\setminus D|$. Hence $|N_{\Gamma_q}(v)\cap D|= 1$ for all $v\in V(\Gamma_q)\setminus D$. Thus $D$ is a perfect dominating set of $\Gamma_q$.
\end{proof}

 \begin{definition}\label{GrGq}
Let $q\ge 4$ be a  prime power and let $x\in \mathbb{F}_q\setminus \{0,1\}$.   Define  $G^x_q$  as the $q$-regular graph of girth 8 and order $2q(q^2-2)$ constructed in Proposition \ref{perfect1}.
\end{definition}

\begin{theorem}\label{q-1}
Let $q\ge 4$ be a  prime power  and let $G^x_q$ be the graph given in Definition~\ref{GrGq}.
Let $R=N_{G^x_q}(\{(\varrho,j,k)_0: j,k\in \mathbb{F}_q, j\not=0,1,x\})\cap N^5_{G^x_q}((\varrho,1,0)_0)$.
Then, the set
 $$ S:= \bigcup_{j\in \mathbb{F}_q} N_{G^x_q}[(\varrho,1,j)_0]\cup   N_{G^x_q}[R]$$
is a perfect dominating set in $G^x_q$ of cardinality $4q^2-6q$.
Hence, $G^x_q-S$ is a $(q-1)$-regular graph of girth $8$ and order $2q(q-1)^2$.
\end{theorem}
\begin{proof} Once $x\in \mathbb{F}_q\setminus \{0,1\}$ has been chosen to define $G^x_q$, to simplify notation, we will denote $G^x_q$ by $G_q$ throughout the proof.
Denote by $P=\{(\varrho,j,k)_0: j,k\in \mathbb{F}_q, j\not=0,1,x\}$, then  $R=N_{G_q}(P)\cap N^5_{G_q}((\varrho,1,0)_0)$. Note that $d_{G_q}((\varrho,1,0)_0, (\varrho,j,k)_0)=4$, because according to Definition \ref{cage}, $G_q$ contains the following paths of length four (see Figure \ref{arbol}):
$(\varrho,1,0)_0~(1,b,0)_1(w,w+b,w+2b )_0~ (j, t,k)_1~ (\varrho,j,k)_0$, for all $b,j,t\in \mathbb{F}_q$ such that $b+w\ne 0$ due to the vertices $(j,0,k)_0$  with second coordinate zero have been removed from $\Gamma_q$ to obtain $G_q$.
 By Definition \ref{cage} we have $w+b=jw+t$ and $w+2b=j^2w+2jt+k$. If  $w+b=0$, then $-w=b=tj^{-1}$ and $b=jt+k$ yielding that  $t=(1-j^2)^{-1}jk$. This implies that
$(j,(1-j^2)^{-1}jk,k)_1\in R$ is the unique neighbor in $R$ of $(\varrho,j,k)_0\in P$. Therefore every $(\varrho,j,k)_0\in P$ has a unique neighbor $(j,t,k)_1\in R$ leading to:
\begin{equation}\label{aux1}
 |R|=|P|=q(q-3).
 \end{equation}
 Thus, every $v\in N_{G_q}(R)\setminus P$ has at most
 $|R|/q=q-3$ neighbors in $R$ because for each $j$ the vertices from the set $\{(\varrho,j,k)_0:  k \in \mathbb{F}_q \}\subset P$  are mutually at distance 6 (they were the $q$ neighbors in $\Gamma_q$ of the removed vertex $(\varrho,\varrho,j)_1$). Furthermore, every $v\in N_{G_q}(R)\setminus P$ has at most one neighbor in $N_{G_q}^5((\varrho,1,0)_0)\setminus R$ because the vertices $\{(\varrho,1,j)_0: j \in \mathbb{F}_q, j\ne 0\}$ are mutually at distance 6. Therefore every $v\in N_{G_q}(R)\setminus P$   has at least two neighbors  in $ N_{G_q}^3((\varrho,1,0)_0)$. Thus denoting $K=N_{G_q}(N_{G_q}(R)\setminus P)\cap N_{G_q}^3((\varrho,1,0)_0)$ we have
 \begin{equation}\label{aux2}  |K|\ge  2 |N_{G_q}(R)\setminus P|  .
 \end{equation}
\noindent Moreover, observe that $ (N_{G_q}(P)\setminus R)\cap K =\emptyset$ because these two sets are at distance four (see Figure \ref{arbol}). Since the elements of $P$ are mutually at distance at least 4 we obtain that $|N_{G_q}(P)\setminus R|=q|P|-|R|=(q-1)|P|$. Hence by (\ref{aux1})
$$ |N_{G_q}^3((\varrho,1,0)_0)| \ge  | N_{G_q}(P)\setminus R |+|K|
=  (q-1)|P|+|K|=(q-1)q(q-3)+|K|.
  $$
Since  $|N_{G_q}^3((\varrho,1,0)_0)|=q(q-1)^2$  we obtain that $|K|\le 2q(q-1)$ yielding by (\ref{aux2}) that  $|N_{G_q}(R)\setminus P|\le q(q-1)$. As $P$ contains at least $q$ elements mutually at distance 6,  $R$ contains at least $q$ elements mutually at distance 4.  Thus we have
  $|N_{G_q}(R)\setminus P|\ge q^2-q$. Therefore $|N_{G_q}(R)\setminus P|= q^2-q$ and all the above inequalities are actually equalities. Thus by (\ref{aux1}) we get
\begin{equation}\label{aux3}|N_{G_q}(R) |=  q^2-q+|P|=2q( q-2)\end{equation}
 and every $v\in N_{G_q}(R)\setminus P$   has exactly $1$ neighbor in  $N_{G_q}^5((\varrho,1,0)_0)\setminus R$.
 Therefore we have
$$\begin{array}{lll}  | N^4_{G_q}((\varrho,1,0)_0)\setminus N_{G_q}(R)|&=&\displaystyle |\bigcup_{j\in \mathbb{F}_q\setminus \{0\}} (N^2_{G_q}((\varrho,1,j)_0)\cup P)\setminus N_{G_q}(R)|\\[1ex]
&=& q(q-1)^2+q(q-3)-2q(q-2)\\ &=& q(q-1)   (q-2).\end{array}$$
 Let us denote by $E[A,B]$ the set of edges between any two sets of vertices $A$ and $B$. Then
$|E[N^3_{G_q}((\varrho,1,0)_0),N^4_{G_q}((\varrho,1,0)_0)]|=q(q-1)^3$ and   $|E[N^3_{G_q}((\varrho,1,0)_0),N^4_{G_q}((\varrho,1,0)_0)\setminus N_{G_q}(R)]|=q(q-1)^2(q-2)$. Therefore, $$|E[N^3_{G_q}((\varrho,1,0)_0),  N_{G_q}(R)]|=q(q-1)^3-q(q-1)^2(q-2)=q(q-1)^2=|N^3_{G_q}((\varrho,1,0)_0)|,$$ which  implies that every $v\in N_{G_q}^3((\varrho,1,0)_0)$   has exactly one neighbor   in $N_{G_q}(R) $. It follows   that  $S=\bigcup_{j\in \mathbb{F}_q} N_{G_q}[(\varrho,1,j)_0]\cup   N_{G_q}[R] $ is a perfect dominating set of $G_q$.
Furthermore, by (\ref{aux1}) and  (\ref{aux3}),  $|S|=q^2+q+ q(3q-7)=4q^2-6q$.
 Therefore a $(q-1)$-regular graph of girth 8  can be obtained by deleting from $G_q$  the  perfect dominating set $S$, see Figure \ref{arbol}. This graph  has order $2q(q^2-2)- 2q(2q -3)=2q(q -1)^2 $.
\end{proof}

\begin{figure}
  \centering
  \includegraphics[width=16cm]{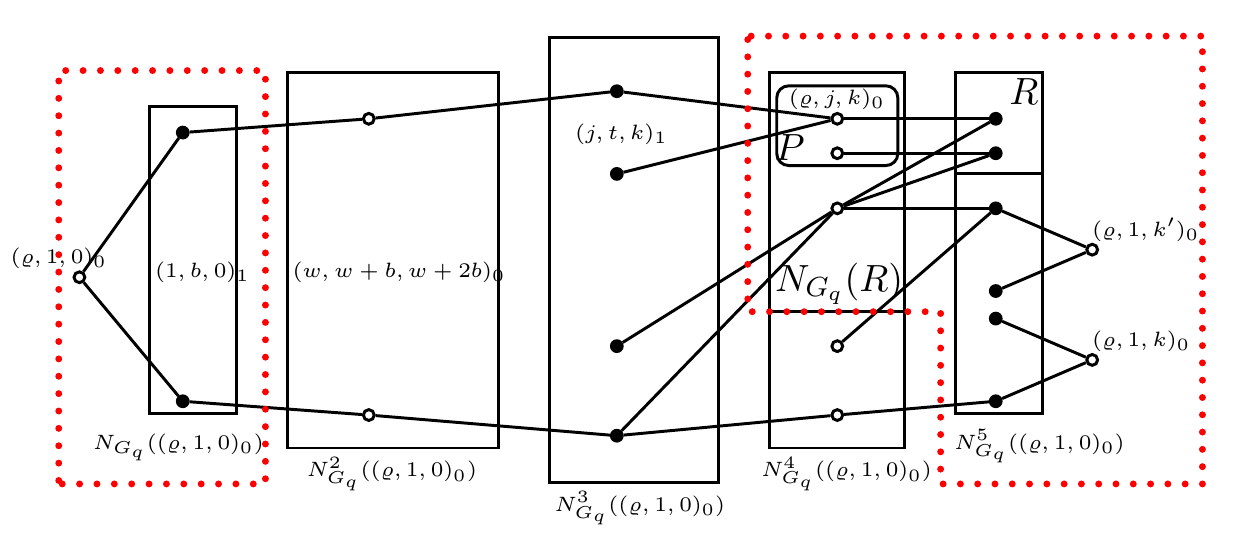}\\
  \caption{Structure of the graph $G_q$. The perfect dominating set lies inside the dotted box.}\label{arbol}
\end{figure}

\end{document}